\documentclass[11pt]{article}
\usepackage{a4}
\usepackage{latexsym}
\usepackage{amsfonts}
\usepackage{amssymb}
\usepackage{amsmath}
\usepackage{amsthm}
\newtheorem{theorem}{Theorem}[section]
\newtheorem{lemma}[theorem]{Lemma}

\newtheorem{corollary}[theorem]{Corollary}

\newtheorem*{claim*}{Claim}

\title{Variations on the sum-product problem}
\author{Brendan Murphy, Oliver Roche-Newton and Ilya Shkredov
\footnote{
The first author was supported by a WUN Researcher Mobility Grant, and would like to thank Bristol University for their hospitality while this research was conducted.
He would also like to thank the students of MTH 440 for their generous support during the Fall 2013 semester.
The second author was
partially supported by the Grant ERC-AdG 321104 and EPSRC Doctoral Prize Scheme (Grant Ref:  EP/K503125/1). 
The third author was supported by
grant mol\underline{ }a\underline{ }ved 12--01--33080,
Russian Government project 11.G34.31.0053,
Federal Program ``Scientific and scientific--pedagogical staff of innovative Russia" 2009--2013
and
grant Leading Scientific Schools N 2519.2012.1. }
}

\def\eps{\varepsilon}
\def\E{\mathsf {E}}

\def\Gr{{\mathbf G}}
\def\D{\Delta}
\def\m{\times}

\def\L{{\cal L}}

\def\R{{\mathbb R}}

\begin{document}
\maketitle

\begin{abstract} This paper considers various formulations of the sum-product problem. It is shown that, for a finite set $A\subset{\mathbb{R}}$,
$$|A(A+A)|\gg{|A|^{\frac{3}{2}+\frac{1}{178}}},$$
giving a partial answer to a conjecture of Balog. In a similar spirit, it is established that
$$|A(A+A+A+A)|\gg{\frac{|A|^2}{\log{|A|}}},$$
a bound which is optimal up to constant and logarithmic factors. We also prove several new results concerning sum-product estimates and expanders, for example, showing that
$$|A(A+a)|\gg{|A|^{3/2}}$$
holds for a typical element of $A$.  
\end{abstract}
\section{Introduction}

Given a finite set $A\subset{\mathbb{N}}$, one can define the \textit{sum set}, and respectively the \textit{product set}, by
$$A+A:=\{a+b:a,b\in{A}\}$$
and
$$AA:=\{ab:a,b\in{A}\}.$$
The Erd\H{o}s-Szemer\'{e}di \cite{ES} conjecture states, for all $\epsilon>0$,
$$\max{\{|A+A|,|AA|\}}\gg{|A|^{2-\epsilon}},$$
and it is natural to extend this conjecture to other fields, particularly the real numbers. 
In this direction, the current state-of-the-art bound, due to Solymosi \cite{solymosi}, states that for any $A\subset{\mathbb{R}}$
\begin{equation}
\max{\{|A+A|,|AA|\}}\gg{\frac{|A|^{4/3}}{(\log{|A|})^{1/3}}}.
\label{soly1}
\end{equation}
When looking to construct a set $A$ which generates a very small sum set $A+A$, one needs to impose an additive structure on $A$, and an additive progression is an example of a highly additively structured set. Similarly, if $A$ has a very small product set, it must be to some extent multiplicatively structured. Loosely speaking, the Erd\H{o}s-Szemer\'{e}di conjecture reflects the intuitive observation that a set of integers, or indeed real numbers, cannot be highly structured in both a multiplicative and additive sense.

In this paper, we consider other ways to quantify this observation. In particular, one would expect that a set will grow considerably under a combination of additive and multiplicative operations. Consider the set

$$A(A+A):=\{a(b+c):a,b,c\in{A}\}.$$

The same heuristic argument as the above leads us to expect that this set will always be large. Indeed, any progress towards the Erd\H{o}s-Szemer\'{e}di conjecture immediately yields a lower bound for the quantity $|A(A+A)|$. To see this, let us assume for simplicity that $0,1\in{A}$. This implies that $AA$ and $A+A$ are subsets of $A(A+A)$, and therefore, Solymosi's result \eqref{soly1} implies that

\begin{equation}
|A(A+A)|\gg{\frac{|A|^{4/3}}{(\log{|A|})^{1/3}}}.
\label{soly2}
\end{equation}

The expectation that $|A(A+A)|$ is always large was formalised by Balog\footnote{This conjecture was made as part of a talk at the conference ``Additive Combinatorics in Paris". A similar conjecture was made in \cite{balog} for the set $A+AA$.} \cite{balog}, who conjectured that, for all $\epsilon>0$,
$$|A(A+A)|\gg{|A|^{2-\epsilon}}.$$
Note that if $A=\{1,2,\cdots,N\}$, then
$$A(A+A)\subset{\{nm:n,m\in{[2N]}\}}.$$
This set obviously has cardinality $O(N^2)$, and in fact it is known that the product set determined by the first $N$ integers has cardinality $o(N^2)$.\footnote{See Ford \cite{ford} for a precise statement concerning the size of this product set.} Therefore, we cannot expect to prove anything stronger than this conjecture.

It is worth pointing out that Balog's conjecture is also close to being sharp in the dual case where $A$ is a geometric progression. Indeed, $A(A+A)\subset{AA+AA}$, and if $AA$ has cardinality $O(|A|)$, then $|AA+AA|=O(|A|^2)$.

By attacking the problem of establishing lower bounds on $|A(A+A)|$ directly (as opposed to applying Solymosi's sum-product estimate rather crudely), it is possible to obtain quantitatively improved results. Using a straightforward application of the Szemer\'{e}di-Trotter theorem\footnote{To the best of our knowledge, a proof of this does not appear in the existing literature. Exercise 8.3.3 in Tao-Vu \cite{TV} observes that $|AA+A|\gg{|A|^{3/2}}$, and this proof can easily be adapted to show that $|A(A+A)|\gg{|A|^{3/2}}$. These simple proofs are similar to those of the earlier sum-product estimates of Elekes \cite{elekes}.}, one can show that

\begin{equation}
|A(A+A)|\gg{|A|^{3/2}}.
\label{trivial}
\end{equation}

The original aim here was to improve on this lower bound, which we do by proving\footnote{The rough inequality $\gtrapprox$ is used here to suppress logarithmic and constant factors. See the forthcoming notation section for a precise definition of the meaning of this symbol.} that

\begin{equation}
|A(A+A)|\gtrapprox{|A|^{\frac{3}{2}+\frac{1}{178}}}.
\label{main}
\end{equation}

Although the method leads only to a small improvement for this problem, it turns out to be much more effective when more variables are involved. To this end we prove the following result:
\begin{equation}
|A(A+A+A+A)|\gg{\frac{|A|^2}{\log{|A|}}}.
\label{A(A+A+A+A)}
\end{equation}

Observe that this bound is tight, up to logarithmic factors, in the case when $A$ is an arithmetic progression. Indeed, the aforementioned work of Ford tells us that some logarithmic factor is necessary here. The set $A(A+A+A+A)$ has similar characteristics to $A(A+A)$, and inequality \eqref{A(A+A+A+A)} proves a weak version of Balog's conjecture.

The main tool in this paper is the Szemer\'{e}di-Trotter theorem, although its application is a little more involved than the straightforward application which gives the bound \sloppy{${|A(A+A)|\gg{|A|^{3/2}}}$}. To be more precise, we use an application of the Szemer\'{e}di-Trotter theorem to establish our main lemma, which bounds the cardinality of $A(A+A)$ in terms of the \textit{multiplicative energy} of $A$. The multiplicative energy, denoted $\E^*(A)$, is the number of solutions to the equation
\begin{equation}
a_1a_2=a_3a_4,
\label{energydefn}
\end{equation}
such that $a_1,a_2,a_3,a_4\in{A}$. This quantity has been an important feature in some of the existing bounds for the sum-product problem (see \cite{solymosi} and \cite{solymosi2}).

Of particular importance in this paper is the forthcoming Lemma \ref{mainlemma}, which gives an improvement to \eqref{trivial} unless the multiplicative energy is almost as large as possible. However, in the case where the multiplicative energy is very large, the Balog-Szemer\'{e}di-Gowers Theorem implies the existence of a large subset $A'\subset{A}$ with the property that the ratio set\footnote{The ratio set $A:A$ determined by $A$ is the set of all pairwise ratios; that is the set $\{a/b:a,b\in{A}\}$.} $A':A'$ is small. We can then use a sum-product estimate from \cite{LiORN2} to get an improvement to \eqref{trivial}. This gives a sketch of the proof of \eqref{main}.

Another variation of the sum-product problem is to consider product sets of additive shifts, which we might expect to be large. It was shown by Garaev and Shen \cite{GS} that for a finite set $A\subset{\mathbb{R}}$, one has $|A(A+1)|\gg{|A|^{5/4}}$, and this bound was improved slightly in \cite{TimORN}. Note that the value $1$ is not important here, and these results hold if $1$ is replaced in the statement by any non-zero value. The problem of determining the best possible lower bound for the size of $A(A+1)$ remains open.

We will prove several bounds which relate to this problem, as well as the problem of finding better lower bounds for $|A(A+A)|$.  
For example, in the forthcoming Theorem \ref{translates1}, it will be established that, for at least half of the elements $a\in{A}$ we have
\begin{equation}
|A(A+a)|\gg{|A|^{3/2}}.
\label{example}
\end{equation}
Note that this result reproves the bound \eqref{trivial}, but using two variables as opposed to three.

\subsection*{Structure of this paper}

The rest of this paper is structured as follows. We conclude this introductory section by explaining some of notation that will be used. In section 2, we give a full list of the new results in this paper. Section 3 gives proofs of the main preliminary results, all of which follow from the Szemer\'{e}di-Trotter theorem. Section 4 provides proofs of the main results - including \eqref{main} and \eqref{A(A+A+A+A)}. In section 5, we prove several other results concerning growth of sets under additive and multiplicative operations; this includes \eqref{example} and several results in a similar spirit.
It will be necessary to call upon some results from earlier works, such as the Szemer\'{e}di-Trotter Theorem and the Balog-Szemer\'{e}di-Gowers Theorem; any such results will be stated as and when they are needed.

\subsection*{Notation}

Throughout the paper, the standard notation
$\ll,\gg$ and, respectively, $O,\Omega$ is applied to positive quantities in the usual way. Saying, $X\gg Y$ or $X=\Omega(Y)$ means that $X\geq cY$, for some absolute constant $c>0$. We write $X\approx{Y}$ if both $X\gg{Y}$ and $X\ll{Y}$. The notation $\gtrapprox$ is occasionally used to suppress both constant and logarithmic factors. To be more precise, we write $X\gtrapprox{Y}$ if there exist positive constants $C$ and $c$ such that $X\geq{c\frac{Y}{(\log{X})^C}}$.  All logarithms in this paper are to base 2.

Let $A,B\subset{\mathbb{R}\setminus{\{0\}}}$ be finite sets\footnote{Note that the assumption that $0\notin{A}$ is merely added to avoid the inconvenience of the possibility of dividing by zero, and simplifies matters slightly. All of the bounds in this paper are unaffected; we may simply start all proofs by deleting zero and apply the analysis to $A':=A\setminus{\{0\}}$, with only the implied constants being affected.}.
We have already defined the sum set $A+B$ and the product set $AB$.
 
The \emph{difference set $A-B$}\/ and the \emph{ratio set $A:B$}\/ are defined by
\[
A-B=\{a-b\colon a\in A,b\in B\}\quad\mbox{and}\quad A:B=\{a/b\colon a\in A,b\in B\}.
\]

Given $x\in{\mathbb{R}}$, we use the notation $r_{A+B}(x)$ to denote the number of representations of $x$ as an element of $A+B$. To be precise
$$r_{A+B}(x):=|\{(a,b)\in{A\times{B}}:a+b=x\}|.$$
This notation will be used flexibly throughout the paper to define the number of representations of $x$ as an element of a given set described in the subscript. For example,
$$r_{A(B+C)}:=|\{(a,b,c)\in{A\times{B}\times{C}}:a(b+c)=x\}|.$$

In a slight generalisation of the earlier definition, the \textit{multiplicative energy} of $A$ and $B$, denoted $\E^*(A,B) = \E^*_2 (A,B)$, is defined to be the number of solutions to the equation
$$a_1b_1=a_2b_2,$$
such that $a_i\in{A}$ and $b_i\in{B}$. This quantity is also the number of solutions to
$$\frac{a_1}{a_2}=\frac{b_2}{b_1}$$
and
$$\frac{a_1}{b_2}=\frac{a_2}{b_1}.$$
Observe that $\E^*(A,B)$ can also be defined in terms of the representation function $r$ as follows:
\begin{align*}
\E^*(A,B)&=\sum_{x}r_{A:B}^2(x)\\
&=\sum_xr_{A:A}(x)r_{B:B}(x)\\
&=\sum_xr_{AB}^2(x).
\end{align*}

We use $\E^*(A)$ as a shorthand for $\E^*(A,A)$.

One of the fundamental basic properties of the multiplicative energy is the following well-known lower bound:
\begin{equation}
\E^*(A,B)\geq{\frac{|A|^2|B|^2}{|AB|}}.
\label{CS}
\end{equation}
The proof is short and straightforward, arising from a single application of the Cauchy-Schwarz inequality. The full details can be seen in Chapter 2 of \cite{TV}.

The above definitions can all be extended in the obvious way to define the \textit{additive energy} of $A$ and $B$, denoted $\E^+(A,B)$. So,

$$\E^+(A,B):=\sum_xr_{A-B}^2(x).$$

The third \textit{moment multiplicative energy} is the quantity
$$\E^*_3(A):=\sum_xr_{A:A}^3(x),$$
and similarly, the \textit{third moment additive energy} is defined by
$$\E^+_3(A):=\sum_xr_{A-A}^3(x).$$

In recent years, third moment energy has played an important role in quantitative progress on various problems in arithmetic combinatorics. See for example \cite{TimORN}, \cite{LiORN2}, \cite{SS1},\cite{SS2},\cite{SS3},
\cite{SV} and \cite{TV}.

We will use the Katz--Koester trick \cite{kk}, which is the observation that
$$
    |(A + A) \cap (A + A - s)| \ge |A + A_{s}| \,,
$$
and
$$
    |(A - A) \cap (A - A - s)| \ge |A - (A\cap (A+s))| \,,
$$
where $A_s = A\cap (A-s)$. We also need the following identity (see \cite{SV}, Corollary 2.5)
\begin{equation}\label{f:A^2-D(A)}
    \sum_s |A \pm A_s| = |A^2 \pm \Delta(A)| \,,
\end{equation}
where
$$
    \Delta(A) = \{ (a,a) ~:~ a\in A \} \,.
$$

\section{Statement of results}

\subsection{Preliminary Results - Applications of the Szemer\'{e}di-Trotter Theorem}

The most important ingredient for the sum-product type results in this paper is the Szemer\'{e}di-Trotter Theorem \cite{ST}:
\begin{theorem} \label{ST1} Let $P\subset{\mathbb{R}^2}$ be a finite set of points and let $L$ be a collection of lines in the real plane. Then
$$I(P,L):=|\{(p,l)\in{P\times{L}}:p\in{l}\}|\ll{|P|^{2/3}|L|^{2/3}+|L|+|P|}.$$
\end{theorem}
Here by $I(P,L)$ we denote the number of incidences between a set of points $P$ and a set of lines $L$.
Given a set of lines $L$, we call a point that is incident to at least $t$ lines of $L$ a \emph{$t$-rich point}, and we let $P_t$ denote the set of all $t$-rich points of $L$. 
The Szemer\'edi-Trotter theorem implies a bound on the number of $t$-rich points:
\begin{corollary} \label{STcor} Let $L$ be a collection of lines in $\mathbb{R}^2$, let $t\geq{2}$ be a parameter and let $P_t$ be the set of all $t$-rich points of $L$.
Then
$$|P_t|\ll{\frac{|L|^2}{t^3}+\frac{|L|}{t}}.$$
Further, if no point of $P_t$ is incident to more than $|L|^{1/2}$ lines, then %
 
\[
|P_t|\ll\frac{|L|^2}{t^3}.
\]
\end{corollary}

This result is used to prove the main preliminary results in this paper, which give us information about various kinds of energies.

\begin{lemma} \label{sum+} Let $A,B$ and $X$ be finite subsets of $\mathbb{R}$ such that $|X|\leq |A||B|$. Then
\[
\sum_{x\in X}\E^+(A,xB)\ll |A|^{3/2}|B|^{3/2}|X|^{1/2}.
\]
\end{lemma}
Note that $\E^+(A,xB)\geq |A||B|$ for all $x$, so the condition $|X|\leq|A||B|$ is necessary.
Bourgain formulated a similar theorem (``Theorem C'' of \cite{bourgain}) for subsets of fields with prime cardinality.
Bourgain's theorem is closely related to the Szemer\'edi-Trotter theorem for finite fields \cite{dvir,HelfgottRudnev}.

This result works in the same way with the roles of addition and multiplication reversed.

\begin{lemma} \label{sum*} Let $A,B$ and $X$ be finite subsets of $\mathbb{R}$ such that $|X|\leq{|A||B|}$. Then
$$\sum_{x\in{X}}\E^*(A,x+B)\ll{|A|^{3/2}|B|^{3/2}|X|^{1/2}}.$$
\end{lemma}

A similar method is used to establish the following important lemma, which will be applied several times in this paper.
\begin{lemma}\label{mainlemma} For any finite sets $A,B,C\subset{\mathbb{R}}$, we have
\[\E^*_2(A)|A(B+C)|^2\gg{\frac{|A|^4|B||C|}{\log{|A|}}}.\]
\end{lemma}
We remark that Lemma \ref{mainlemma} is optimal, up to logarithmic factors, in the case when $A=B=C=\{1,\cdots,N\}$.

\subsection{Main Results}

The next two theorems represent the main results in this paper. Although they were mentioned in the introduction, they are restated here for the completeness of this section.

\begin{theorem} \label{main1} Let $A\subset{\mathbb{R}}$ be a finite set. Then
$$|A(A+A)|\gtrapprox{|A|^{\frac{3}{2}+\frac{1}{178}}}.$$
\end{theorem}

\begin{theorem} \label{main2} Let $A\subset{\mathbb{R}}$ be a finite set. Then
$$|A(A+A+A+A)|\gg{\frac{|A|^2}{\log{|A|}}}.$$
\end{theorem}

We also prove the following suboptimal result, which is closely related to Theorems \ref{main1} and \ref{main2}:

\begin{theorem} \label{main3} Let $A\subset{\mathbb{R}}$ be a finite set. Then
\begin{equation}
|A(A+A+A)|\gtrapprox{|A|^{\frac{7}{4}+\frac{1}{284}}}.
\label{A(A+A+A)}
\end{equation}
\end{theorem}

\subsection{Products of Additive Shifts}

We will prove a family of results bounding from below the product set of translates of a set $A$. One may observe a familiar gradient in this sequence of results: the bounds improve as we introduce more variables and more translates. It was proven in \cite{TimORN} that, for any finite set $A\subset{\mathbb{R}}$ and any value $x\in{\mathbb{R}\setminus{\{0\}}}$,
\begin{equation}
|A(A+x)|\gg{\frac{|A|^{24/19}}{(\log{|A|})^{2/19}}}.
\label{old}
\end{equation}

As mentioned in the introduction, we will prove the following Theorem, which shows that we can usually improve on \eqref{old} in the case when $x\in{A}$.

\begin{theorem} \label{translates1}
Let $A\subset{\mathbb{R}}$ be a finite set. Then there exists a subset $A'\subset{A}$, such that $|A'|\geq{\frac{|A|}{2}}$, and for all $a\in{A'}$,
$$|A(A+a)|\gg{|A|^{3/2}}.$$
\end{theorem}

Adding more variables to our set leads to better lower bounds:

\begin{theorem} \label{translates2}
Let $A\subset\mathbb{R}$ be a finite set. Then there exists a subset $A'\subset{A}$ with cardinality $|A'|\geq{\frac{|A|}{2}}$, such that for all $a\in{A'}$,
$$|(A+A)(A+a)|\gg{\frac{|A|^{5/3}}{(\log{|A|})^{1/3}}}.$$
\end{theorem}

Theorem \ref{translates2} is similar to the result of Theorem \ref{main1}, especially if we think of the set $A(A+A)$ in the terms $(A+0)(A+A)$. This result tells us that we can usually do better than Theorem \ref{main1} if $0$ is replaced by an element of $A$.

The next theorem is quantitatively worse than Theorem \ref{translates2}, but is more general, since it applies not only for most $a\in{A}$, but to all real numbers except for a single problematic value.
\begin{theorem} \label{translates3}
Let $A\subset\mathbb{R}$ be a finite set. Then, for all but at most one value $x\in{\mathbb{R}}$,
\begin{equation}
|(A+A)(A+x)|\gg{\frac{|A|^{11/7}}{(\log{|A|})^{3/7}}}.
\label{annoying}
\end{equation}
\end{theorem}
Unfortunately, this does not lead to an improvement to Theorem \ref{main1}, since the single bad $x$ that violates \eqref{annoying} may be equal to zero.

\subsection{Further results}

Finally, we formulate a theorem of a slightly different nature. 
\begin{theorem} 
    Let $A,B\subseteq \R$ be finite sets.

Then
\begin{equation}\label{f:main_intr_2_new}
    |A+B|^3 \gg \frac{|B| \E^{*} (A)}{\log |A|} \ge \frac{|A|^4 |B|}{|AA^{\pm1}| \log |A|} \,,
\end{equation}

and
\begin{equation}\label{f:main_intr_2'-_new}
    |B+AA|^3 \gg \frac{|B|^{} |A|^{12}}{(\E^*_3 (A))^2 |AA^{-1}| \log |A|} \,.
\end{equation}
\label{t:main_intr_II}
\end{theorem}

Let us say a little about the meaning of these two bounds. If we fix $A=B$, then \eqref{f:main_intr_2_new} tells us that $|AA|$ is very large if $|A+A|$ is very small. Similar results are already known; for example, a quantitatively improved version of this statement is a consequence of Solymosi's sum-product estimate in \cite{solymosi}. The benefit of \eqref{f:main_intr_2_new} is that it also works for a mixed sum set $A+B$.

One of the main objectives of this paper is to study the set $A(A+A)$, and inequality \eqref{f:main_intr_2'-_new} considers the dual problem of the set $A+AA$. As stated earlier, it is easy to show that $|A+AA|\gg{|A|^{3/2}}$. If we fix $A=B$ in \eqref{f:main_intr_2'-_new}, then this bound gives an improvement in the case when $\E_3^*(A)$ is small. We hope to carry out a more detailed study of the set $A+AA$ in a forthcoming paper.

\section{Proofs of Preliminary Results}

\subsection*{Proof of Lemma \ref{sum+}}

Recall that Lemma \ref{sum+} states that for $|X|\leq|A||B|$,
\[
\sum_{x\in X}\E^+(A,xB)\ll |A|^{3/2}|B|^{3/2}|X|^{1/2}.
\]
 
Note that
\begin{equation}
\sum_{x\in{X}}\E^+(A,xB)=\sum_{x\in{X}}\sum_yr_{A+xB}^2(y).
\label{firstsum}
\end{equation}

We will interpret $r_{A+xB}(y)$ geometrically and use corollary \ref{STcor} to show that there are not too many pairs $(x,y)$ for which the quantity $r_{A+xB}(y)$ is large.

\begin{claim*}
Let $R_t=\{(x,y):r_{A+xB}(y)\geq{t}\}$.
Then for any integer $t\geq 2$,
\begin{equation}
\label{A+xA}
  |R_t|\ll{\frac{|A|^2|B|^2}{t^3}}.
\end{equation}
\end{claim*}
\begin{proof}[Proof of Claim]
Define a collection of lines
$$L:=\{l_{a,b}:(a,b)\in{A\times{B}}\},$$
where $l_{a,b}$ is the line with equation $y=ax+b$.
Clearly, $|L|=|A||B|$.

Since $r_{A+xB}(y)$ counts the number of solutions $(x,y)$ to $y=ax+b$, we see that $r_{A+xB}(y)$ is the number of lines of $L$ that are incident to $(x,y)$.
Thus every pair $(x,y)$ in $R_t$ is a $t$-rich point of $L$.
Further, because
\[r_{A+xB}(y)\leq{\min{\{|A|,|B|\}}}\leq{(|A||B|)^{1/2}}\]
there are no pairs $(x,y)$ such that $r_{A+xB}(y) > (|A||B|)^{1/2}$; that is, there are no points incident to more than $|L|^{1/2}$ lines of $L$.
It follows from Corollary \ref{STcor} that
\[
|R_t|\leq |P_t|\ll\frac{|L|^2}{t^3}=\frac{|A|^2|B|^2}{t^3},
\]
which proves the claim.
\end{proof}

Now we will interpolate between \eqref{A+xA} and a trivial bound.
Let $\triangle\geq{1}$ be an integer to be specified later. The sum in \eqref{firstsum} can be divided up as follows:
\begin{align}
\sum_{x\in{X}}\E^+(A,xB)&=\sum_{x\in{X}}\sum_yr_{A+xB}^2(y)
\\&\leq{\sum_{x\in{X}}\sum_{y\,:\,r_{A+xB}(y) \leq \triangle}r_{A+xB}^2(y)+\sum_{(x,y)\,:\,r_{A+xB}(y)>{\triangle}}r_{A+xB}^2(y)}. \label{eq1}
\end{align}

To bound the first term in \eqref{eq1}, observe that
\begin{align}
\sum_{x\in{X}}\sum_{y\,:\,r_{A+xB}(y)\leq \triangle}r_{A+xB}^2(y)&\leq{\triangle\sum_{x\in{X}}\sum_{y}r_{A+xB}(y)}
\\&=\triangle|A||B|\sum_{x\in{X}}1
\\&=|A||B|\triangle|X|. \label{firstterm}
\end{align}

To bound the second term in \eqref{eq1}, we decompose dyadically and then apply \eqref{A+xA} to bound the size of the dyadic sets we are summing over:

\begin{align}
\sum_{(x,y)\,:\,r_{A+xB}(y)>{\triangle}}r_{A+xB}^2(y)&=\sum_{j\geq{1}}\, \sum_{(x,y)\,:\,\triangle2^{j-1}<{r_{A+xB}(y)}\leq\triangle2^j}r_{A+xB}^2(y)
\\&\ll{\sum_{j\geq{1}}\frac{|A|^2|B|^2}{(\triangle2^j)^3}(\triangle2^j)^2}
\\&=\frac{|A|^2|B|^2}{\triangle}\sum_{j\geq{1}}\frac{1}{2^j}
\\&=\frac{|A|^2|B|^2}{\triangle}. \label{secondterm}
\end{align}
For an optimal choice, set the parameter $\triangle=\left\lceil\frac{|A|^{1/2}|B|^{1/2}}{|X|^{1/2}}\right\rceil\approx{\frac{|A|^{1/2}|B|^{1/2}}{|X|^{1/2}}}>1$.
The approximate equality here is a consequence of the assumption $\frac{|A|^{1/2}|B|^{1/2}}{|X|^{1/2}}>1$.
 
Combining the bounds from \eqref{firstterm} and \eqref{secondterm} with \eqref{eq1}, it follows that
$$\sum_{x\in{X}}\E^+(A,xB)\ll{|A|^{3/2}|B|^{3/2}|X|^{1/2}},$$
as required.

This completes the proof of Lemma \ref{sum+}.
\begin{flushright}
\qedsymbol
\end{flushright}

The proof of Lemma \ref{sum*} is essentially the same, with the roles of addition and multiplication reversed. For completeness, a full proof is provided.

\subsection*{Proof of Lemma \ref{sum*}} 

Recall that Lemma \ref{sum*} states that for $|X|\leq |A||B|$, 
\[
\sum_{x\in X}\E^*(A,B+x)\ll |A|^{3/2}|B|^{3/2}|X|^{1/2}.
\]

Define a set of lines $L:=\{l_{a,b}:(a,b)\in{A\times{B}}\}$, where $l_{a,b}$ now represents the line with equation $y=a(b+x)$.
These lines are all distinct and so $|L|=|A||B|$.
Since $r_{A(B+x)}(y)$ is the number of such lines incident to a point $(x,y)$, we can apply Corollary \ref{STcor} and argue as before to show that
\begin{equation}
|\{(x,y):r_{A(B+x)}(y)\geq{t}\}|\ll{\frac{|A|^2|B|^2}{t^3}},
\label{A(x+A)}
\end{equation}
for any integer $t\geq{1}$.

Next, we use the bound \eqref{A(x+A)} in the following calculation, which holds for any integer $\triangle>1$:
\begin{align*}
\sum_{x\in{x}} \E^*(A,B+x)&=\sum_{x\in{X}}\sum_{y}r_{A(B+x)}^2(y)
\\&\leq{\sum_{x\in{X}}\sum_{y\,:\,r_{A(B+x)}(y)\le \triangle}r_{A(B+x)}^2(y)+\sum_{(x,y)\,:\,r_{A(B+x)}(y)>{\triangle}}r_{A(B+x)}^2(y)}
\\&\le \sum_{x\in{X}}\triangle\sum_{y}r_{A(B+x)}(y)+\sum_{j\geq{1}}\sum_{(x,y)\,:\,\triangle2^{j-1}<{r_{A(B+x)}(y)}\le \triangle2^j}r_{A(B+x)}^2(y)
\\&\ll{|A||B||X|\triangle+\sum_{j\geq{1}}(2^j\triangle)^2\frac{|A|^2|B|^2}{(2^j\triangle)^3}}
\\&=|A||B||X|\triangle+\frac{|A|^2|B|^2}{\triangle}.
\end{align*}
If we set $\triangle:=\left\lceil\frac{|A|^{1/2}|B|^{1/2}}{|X|^{1/2}}\right\rceil\approx{\frac{|A|^{1/2}|B|^{1/2}}{|X|^{1/2}}}>1$, the proof is complete.
\begin{flushright}
\qedsymbol
\end{flushright}

We observe the following Corollary of Lemmas \ref{sum+} and \ref{sum*}.
Equation \eqref{f:reformulation_1} is sharp when $A$ is arithmetic progression, which shows that Lemma \ref{sum+} is sharp when $A$ and $B$ are the same arithmetic progression, for a suitable choice of $X$.
 
\begin{corollary}
For any $A\subset {\mathbb{R}}$, we have
\begin{equation}\label{f:reformulation_1}
    \left| \frac{A-A}{A-A} \right| \gg |A|^2 \,.
\end{equation}
\begin{equation}\label{f:reformulation_2}
    \left| \left\{\frac{a_2b_2-a_1b_1}{a_1-a_2}:a_1,a_2\in A,\, b_1,b_2\in A  \right\} \right| \gg |A|^2 \,.
\end{equation}
 
\begin{equation}\label{f:reformulation_3}
    |A-A|^3 \cdot \left| \frac{A-A}{A-A} \right|^{1/2} \gg |A^2 - \Delta(A)|^2 \,.
\end{equation}
\label{c:reformulation}
\end{corollary}
\begin{proof}
Let $X(u)$ denote the indicator function on $X$.  
The statements of Lemma \ref{sum+} and Lemma \ref{sum*} can be written as
\begin{equation}\label{tmp:13.10.2013_1}
    \sum_{x,y} r_{A-A} (x) r_{B-B} (y) X(x/y) \ll |A|^{3/2} |B|^{3/2} |X|^{1/2}
\end{equation}
and
\begin{equation}\label{tmp:13.10.2013_2}
    \sum_{a_1,a_2\in A,\, b_1,b_2\in B} X\left( \frac{a_2 b_2 - a_1 b_1}{a_1-a_2} \right) \ll |A|^{3/2} |B|^{3/2} |X|^{1/2}
\end{equation}
respectively, provided that $|X| \le |A||B|$.
Putting $B=A$ and $X=(A-A)/(A-A)$ into \eqref{tmp:13.10.2013_1} proves \eqref{f:reformulation_1}.
Similarly, putting $B=A$ and \[X=\left\{\frac{a_2 b_2 - a_1 b_1}{a_1-a_2}:a_1,a_2\in A,\, b_1,b_2\in B \right\}\] into \eqref{tmp:13.10.2013_2}, we obtain (\ref{f:reformulation_2}).

Let $D=A-A$
Taking $A=B= D$, $X=D/D$, summing just over $x,y\in D$ in (\ref{tmp:13.10.2013_1}), and using Katz--Koester trick as well as identity (\ref{f:A^2-D(A)}), we get
$$
    |A-A|^3 \cdot \left| \frac{A-A}{A-A} \right|^{1/2}
        \gg
            \left( \sum_{x\in D} r_{D-D} (x) \right)^2
                \ge
                     \left( \sum_{x\in D} |A-A_x| \right)^2
                        =
                            |A^2 - \D(A)|^2
$$
which coincides with (\ref{f:reformulation_3}).
\end{proof}

Inequality (\ref{f:reformulation_1})  can also be deduced from Beck's Theorem, which states that a set of $N$ points in the plane which does not have a single very rich line, will determine $\Omega(N^2)$ distinct lines. See Exercise 8.3.2 in \cite{TV}. A geometric result of Ungar \cite{ungar}, concerning the number of different directions determined by a set of points in the plane, also yields \eqref{f:reformulation_1} as a corollary. Although the result here is not new, it has been stated in order to illustrate the sharpness of Lemma \ref{sum+}. Similar results to \eqref{f:reformulation_2} were established in \cite{tim}; it seems likely that \eqref{f:reformulation_2} is suboptimal.

\subsection*{Proof of Lemma \ref{mainlemma}}

Recall that Lemma \ref{mainlemma} states that
$$\E^*(A)|A(B+C)|^2\gg{\frac{|A|^4|B||C|}{\log{|A|}}}.$$

Let $S^\star$ denote the number of solutions to the equation
\begin{equation}
a_1(b_1+c_1)=a_2(b_2+c_2)\not=0,
\label{solutions}
\end{equation}
such that $a_i\in{A}$, $b_i\in{B}$ and $c_i\in{C}$.
This proof uses a familiar strategy: in order to show that a given set is large, show that there cannot be too many solutions to a particular equation.
The easy part is to bound $S^\star$ from below, using an elementary application of the Cauchy-Schwarz inequality.
First note that
\[
\sum_{x\in A(B+C)}r_{A(B+C)}(x)=|A||B||C|.
\]
Since there are at most $|A||B\cap -C|+|B||C|$ solutions to $a(b+c)=0$, we have
\[
\sum_{x\in A(B+C)\setminus\{0\}}r_{A(B+C)}(x)\geq\frac 12|A||B||C|.
\]
Now we apply the Cauchy-Schwarz inequality:
\begin{align}
\frac 14(|A||B||C|)^2&\leq\left(\sum_{x\in{A(B+C)\setminus\{0\}}}r_{A(B+C)}(x)\right)^2
\\&\leq{|A(B+C)|\sum_{x\not=0}r_{A(B+C)}^2(x)}
\\&=|A(B+C)|S^\star. \label{triv}
\end{align}
The rest of the proof is concerned with finding a satisfactory upper bound for the quantity $S^\star$.
We will eventually conclude that
\begin{equation}
\label{aim}
 S^\star\ll{\E^*(A)^{1/2}|B|^{3/2}|C|^{3/2}(\log{|A|})^{1/2}}.
\end{equation}

If this is proven to be true, one can combine the upper and lower bounds on $S^\star$ from \eqref{aim} and \eqref{triv} respectively, and then a simple rearrangement completes the proof of the lemma.

It remains to prove \eqref{aim}.
To do this, first observe that \eqref{solutions} can be rewritten in the form
\[
\frac{a_1}{a_2}=z=\frac{b_2+c_2}{b_1+c_1}.
\]
Note that we can divide by $b_1+c_1$ because we excluded $0$ in \eqref{solutions}.
If we set $Q=(B+C)/(B+C)$ and \[r_Q(z)=|\{(b_1,b_2,c_1,c_2)\in{B\times{B}\times{C}\times{C}}: z=(b_2+c_2)/(b_1+c_1)\}|,\] then
\[
S^\star=\sum_{z\in (A:A)\cap Q} r_{A:A}(z)r_Q(z).
\]
Applying Cauchy-Schwarz, we have
\begin{equation}
  \label{CS 2}
  S^\star\leq \left(\sum_{z\in A:A} r_{A:A}^2(z)\right)^{1/2}\left(\sum_{z\in A:A} r_Q^2(z)\right)^{1/2}=\E_2^*(A)^{1/2}\left(\sum_{z\in A:A} r_Q^2(z)\right)^{1/2}.
\end{equation}

We will bound the RHS of \eqref{CS 2} using the following distributional estimate on $r_Q(z)$:
\begin{claim*}
  Let $Z_t=\{z\colon r_Q(z)\geq t\}$.
Then for all $t\geq 1$,
\[
|Z_t|\ll\frac{|B|^3|C|^3}{t^2}.
\]
\end{claim*}
If we assume this claim, then by dyadic decomposition:
 \begin{align*}
   \sum_{z\in A:A}r_Q^2(z)&\approx \sum_{j\geq 1}^{\lceil\log|A:A|\rceil}2^{2j}|Z_{2^j}|\\
 &\ll \sum_{j\geq 1}^{\lceil\log|A:A|\rceil}2^{2j}\frac{|B|^3|C|^3}{2^{2j}}\leq 2|B|^3|C|^3\log|A|.
 \end{align*}
Combining this with \eqref{CS 2} yields the desired bound on $S^\star$:
\[
S^\star\ll \E_2^*(A)^{1/2}|B|^{3/2}|C|^{3/2}(\log|A|)^{1/2}.
\]
This concludes the proof of Lemma \ref{mainlemma}, pending the claim.
\begin{flushright}
\qedsymbol
\end{flushright}

Now we will prove the claimed estimate for the distribution of $r_Q(z)$.
\begin{proof}[Proof of Claim]
  First we will get an easy estimate for $|Z_t|$ from Markov's inequality.
Since\footnote{$r_Q(z)$ is supported on $Q$, so if $t\geq 1$ we have $Z_t\subseteq Q$.}
\[
t|Z_t|\leq\sum_{z\in Z_t}r_Q(z)\leq\sum_{z\in Q}r_Q(z)=|B|^2|C|^2,
\]
we have
\begin{equation}
  \label{r_Q Markov}
  |Z_t|\leq\frac{|B|^2|C|^2}t.
\end{equation}
Note that if $|Z_t|\geq |B||C|$, then it follows from \eqref{r_Q Markov} that $t\leq |B||C|$.
But then
\[
\frac{|B|^2|C|^2}t\leq\frac{|B|^3|C|^3}{t^2},
\]
so we have proved the claim in the case $|Z_t|\geq |B||C|$.

Now we will prove the claim when $|Z_t|\leq |B||C|$ using Lemma \ref{sum+}.
To do this we make a key observation, which is inspired by the Elekes-Sharir set-up from \cite{rectangles}:
every solution of the equation
\[z=\frac{b_2+c_2}{b_1+c_1}\]
 is a solution to the equation
\[
b_2-zc_1=zb_1-c_2=y
\]
for some $y$.
Thus
\[
r_Q(z)\leq\sum_y r_{zB-C}(y)r_{B-zC}(y).
\]
By the arithmetic-geometric mean inequality
\[r_{zB-C}(y)r_{B-zC}(y)\leq\frac{r_{zB-C}^2(y)+r_{B-zC}^2(y)}2,\]
so
\[
r_Q(z)\leq \frac{\E^+(zB,-C)+\E^+(B,-zC)}2.
\]
Now if $|Z_t|\leq |B||C|$, we can sum over $Z_t$ and apply Lemma \ref{sum+}:
\[
t|Z_t|\leq\sum_{z\in Z_t}r_Q(z)\leq\frac 12\sum_{z\in Z_t}\E^+(zB,-C)+\frac 12\sum_{z\in Z_t}\E^+(B,-zC)\ll |B|^{3/2}|C|^{3/2}|Z_t|^{1/2}.
\]
Rearranging yields the estimate
\[
  |Z_t|\ll\frac{|B|^3|C|^3}{t^2},
\]
as claimed.
\end{proof}

We remark here that this is not the only proof we have found of Lemma \ref{mainlemma} during the process of writing this paper. In particular, it is possible to write a ``shorter" proof which is a relatively straightforward application of an upper bound from \cite{rectangles} on the number of solutions to the equation
$$(a_1-b_1)(c_1-d_1)=(a_2-b_2)(c_2-d_2),$$
such that $a_i\in{A},\cdots,d_i\in{D}$.

Although this proof may appear to be shorter, it relies on the bounds from \cite{rectangles}, which in turn rely on the deeper concepts used by Guth and Katz \cite{GK} in their work on the Erd\H{o}s distinct distance problem. For this reason, we believe that this proof is the more straightforward option.
In addition, this approach leads to better logarithmic factors and works over the complex numbers (see the discussion at the end of the paper).

The following corollary gives an analogous result for third moment multiplicative energy, however, unlike Lemma \ref{mainlemma}, this result does not appear to be optimal.

\begin{corollary}\label{E3} For any finite sets $A,B,C\subset{\mathbb{R}}$, we have

$$\E^*_3(A)|A(B+C)|^4\gg{\frac{|A|^{6}|B|^2|C|^2}{(\log{|A|})^2}}.$$
\end{corollary}
\begin{proof}
By the Cauchy-Schwarz inequality,
\begin{align*}
\sum_xr_{A:A}^2(x)&=\sum_xr_{A:A}^{3/2}(x)r_{A:A}^{1/2}(x)
\\&\leq{\left(\sum_xr_{A:A}^3(x)\right)^{1/2}\left(\sum_xr_{A:A}(x)\right)^{1/2}}
\\&=\left(\E_3^*(A)\right)^{1/2}|A|,
\end{align*}
so that $(\E^*(A))^2\leq{\E^*_3(A)|A|^2}$.

Meanwhile, Lemma \ref{mainlemma} gives $(\E^*(A))^2|A(B+C)|^4\gg{\frac{|A|^8|B|^2|C|^2}{(\log{|A|})^2}}$. Comparing these two bounds gives the desired result.
\end{proof}

\section{Proofs of Main Results}

The first task in this section is to prove Theorem \ref{main1}. This will require an application of the Balog-Szemer\'{e}di-Gowers Theorem. Following the conventional notation $\Gr$ represents a group, whose operation here is written additively\footnote{For our purposes the role of $\Gr$ will usually be played by the set of non-zero real numbers under multiplication.}, and $\E^+(A)$ has the same meaning as was given in section 1. We will need the following result.

\begin{theorem}
    Let $A\subseteq \Gr$ be a set, $K\geq{1}$ and $\E^+(A)\geq{\frac{|A|^{3}}{K}}$.
    Then there is $A'\subseteq A$ such that
    $$
        |A'| \gtrapprox{\frac{|A|}{K}}\,,
    $$
    and
    $$
        |A'-A'| \lessapprox K^4 \frac{|A'|^3}{|A|^2} \,.
    $$
\label{BSG}
\end{theorem}

We remark that the first preprint of this paper used a different version of the Balog-Szemer\'{e}di-Gowers Theorem, due to Schoen \cite{schoen_BSzG}. Shortly after uploading this, we were informed by M.~Z.~Garaev of a quantitatively improved version of the Balog-Szemer\'{e}di-Gowers Theorem, in the form of Theorem \ref{BSG}. This leads to a small improvement in the statement of Theorem \ref{main1}, since our earlier result had an exponent of $\frac{3}{2}+\frac{1}{234}$. The proof of Theorem \ref{BSG} result is short, arising from an application of Lemmas 2.2 and 2.4 in \cite{BG}. It is possible that further small improvements can be made to Theorem \ref{main1} by combining more suitable versions of the Balog-Szemer\'{e}di-Gowers Theorem with our approach.

We will also need a sum-product estimate which is effective in the case when the product set or ratio set is relatively small. The best bound for our purposes is the following\footnote{In the notation of \cite{LiORN2}, we apply this bound with $C=-f(A)$ and $f(x):=\log{x}$.} (see \cite{LiORN2}, Theorem 1.2):

\begin{theorem} \label{convex} Let $A\subset{\mathbb{R}}$. Then
$$|A:A|^{10}|A+A|^9\gtrapprox{|A|^{24}}.$$
\end{theorem}

\subsection*{Proof of Theorem \ref{main1}}
Recall that Theorem \ref{main1} states that  
$$|A(A+A)|\gtrapprox{|A|^{\frac{3}{2}+\frac{1}{178}}}.$$
Write $\E^*(A)=\frac{|A|^3}{K}$. Applying Lemma \ref{mainlemma} with $A=B=C$, it follows that
$$\frac{|A|^3}{K}|A(A+A)|^2\gtrapprox{|A|^6},$$
and so
\begin{equation}
|A(A+A)|\gtrapprox{|A|^{3/2}K^{1/2}}.
\label{lb1}
\end{equation}
On the other hand, by Lemma \ref{BSG}, there exists a subset $A'\subset{A}$ such that
\begin{equation}
|A'|\gg{\frac{|A|}{K}}
\label{subset1}
\end{equation}
and
\begin{equation}
|A':A'|\lessapprox{K^4\frac{|A'|^3}{|A|^2}}.
\label{subset2}
\end{equation}
Now, Theorem \ref{convex} can be applied, and this states that
$$|A':A'|^{10}|A'+A'|^9\gtrapprox{|A'|^{24}}.$$
Applying \eqref{subset2}, it follows that
$$\frac{|A'|^{30}}{|A|^{20}}K^{40}|A'+A'|^9\gtrapprox{|A'|^{24}},$$
so that after rearranging, and applying the crude bound $|A'|\leq{|A|}$, we obtain
$$K^{40}|A'+A'|^9\gtrapprox{\frac{|A|^{20}}{|A'|^{6}}}\geq{|A|^{14}}$$
Using another crude bound,
\begin{equation}\label{f:crude}
    |A(A+A)| \geq |A+A| \geq |A'+A'|,
\end{equation}
yields
\begin{equation}
|A(A+A)|\gtrapprox{\frac{|A|^{14/9}}{K^{40/9}}}.
\label{lb2}
\end{equation}
Finally, we note that the worst case occurs when $K\approx{|A|^{\frac{1}{89}}}$. If $K\geq{|A|^{\frac{1}{89}}}$, then \eqref{lb1} implies that
$$|A(A+A)|\gtrapprox{|A|^{3/2}K^{1/2}}\geq{|A|^{\frac{3}{2}+\frac{1}{178}}},$$
whereas, if $K\leq{|A|^{\frac{1}{89}}}$, one can check that \eqref{lb2} tells us
$$|A(A+A)|\gtrapprox{\frac{|A|^{14/9}}{K^{40/9}}}\geq{|A|^{\frac{3}{2}+\frac{1}{178}}}.$$
We have checked that $|A(A+A)|\gg{|A|^{\frac{3}{2}+\frac{1}{178}}}$ holds in all cases, and so the proof of Theorem \ref{main1} is complete.
\begin{flushright}
\qedsymbol
\end{flushright}

Let us show that the main result can be refined to obtain
\begin{equation}\label{f:main_pred_1_new}
    |A(A+A)| \gg |A|^{\frac{3}{2}+\frac{1}{178}+\eps_0} \,,
\end{equation}
where $\eps_0>0$ is an absolute constant.
To do this we need in an asymmetric version of Balog--Szemer\'{e}di--Gowers theorem, see \cite{TV}, Theorem 2.35.

\begin{theorem}
    Let $A,B\subseteq \Gr$ be two sets, $|B| \le |A|$, and $M\ge 1$ be a real number.
    Let also $L=|A|/|B|$ and $\eps \in (0,1]$ be a real parameter.
    Suppose that
\begin{equation}\label{cond:BSzG_as}
    \E (A,B) \ge \frac{|A| |B|^2}{M} \,.
\end{equation}
    Then there are two sets $H\subseteq \Gr$, $\L \subseteq \Gr$ and $z\in \Gr$ such that
\begin{equation}\label{f:BSzG_as_1}
    |(H+z) \cap B| \gg_\eps M^{-O_\eps (1)} L^{-\eps} |B| \,,
    \quad \quad
    |\L| \ll_{\eps} M^{O_\eps (1)} L^\eps \frac{|A|}{|H|} \,,
\end{equation}
\begin{equation}\label{f:BSzG_as_2}
    |H - H| \ll_{\eps} M^{O_\eps (1)} L^\eps \cdot |H| \,,
\end{equation}
     and
\begin{equation}\label{f:BSzG_as_3}
    |A\cap (H+\L)| \gg_{\eps} M^{-O_\eps (1)} L^{-\eps} |A| \,.
\end{equation}
\label{t:BSzG_as}
\end{theorem}


{\it Proof of inequality (\ref{f:main_pred_1_new}).}
To get $\eps_0$ we need to improve (\ref{f:crude}), that is to show
$|A(A+A)| \ge |A+A|^{1+\eps}$, where $\eps >0$ is some (other) absolute constant.
Suppose not, then $\E^{*} (A,A+A) \gg_M |A|^2 |A+A|$, where $M=|A|^{\eps}$.
Using Theorem \ref{t:BSzG_as} with $B=A$, $A=A+A$, we find, in particular,
$H\subseteq A$ such that $|H| \gg_{M} |A|$ and
$|H H^{-1}| \ll_M |H|$.
Applying Theorem \ref{convex}, we see that
$$
    |A(A+A)| \ge |A+A| \ge |H+H| \gg_M |A|^{14/9} \,.
$$
This completes the proof. $\hfill\Box$

As one can see, the number $\eps_0$ from (\ref{f:main_pred_1_new}) is a result of using
of the asymmetric version of Balog--Szemer\'{e}di--Gowers theorem, and thus is
rather small.

Note that the sum-product estimates in \cite{LiORN2} are quantitatively better when the sum set is replaced by the difference set $A-A$. To be precise, it is proven in \cite{LiORN2} that
$$|A:A|^6|A-A|^5\gg{\frac{|A|^{14}}{(\log{|A|})^2}}.$$
Therefore, the argument of the proof of Theorem \ref{main1} outputs a slightly better bound for the set $A(A-A)$. One can check that
\begin{equation}
|A(A-A)|\gtrapprox{|A|^{\frac{3}{2}+\frac{1}{106}}}.
\label{A(A-A)}
\end{equation}
Again, the asymmetric version of the Balog-Szemer\'{e}di-Gowers Theorem can then be used as above to prove that
$$|A(A-A)|\gg{|A|^{\frac{3}{2}+\frac{1}{106}+\eps_0}},$$
where $\eps_0>0$ is an absolute constant.

\subsection*{Proof of Theorem \ref{main2}}
Recall that Theorem \ref{main2} states that  
$$|A(A+A+A+A)|\gg{\frac{|A|^2}{\log{|A|}}}.$$
The essential step in Solymosi's \cite{solymosi} work on the sum-product problem was to obtain an upper bound on the multiplicative energy in terms of the sum set, as follows:
\begin{equation}
\E^*(A)\ll{|A+A|^2\log{|A|}}.
\label{soly3}
\end{equation}

We mention this bound explicitly because it will now be used in the proof of Theorem \ref{main2}.

Apply Lemma \ref{mainlemma} with $B=C=A+A$. This implies that
$$\E^*(A)|A(A+A+A+A)|^2\gg{\frac{|A+A|^2|A|^4}{\log{|A|}}}.$$
Applying the upper bound \eqref{soly3} on $\E^*(A)$ and then rearranging yields
$$|A(A+A+A+A)|\gg{\frac{|A|^2}{\log{|A|}}}.$$
\begin{flushright}
\qedsymbol
\end{flushright}

\subsection*{Proof of Theorem \ref{main3}}
Recall that Theorem \ref{main3} states that  
$$|A(A+A+A)|\gtrapprox{|A|^{\frac{7}{4}+\frac{1}{284}}}.$$
For the ease of the reader, we begin by writing down a short proof of the fact that
\begin{equation}
|A(A+A+A)|\gtrapprox{\frac{|A|^{7/4}}{(\log{|A|})^{3/4}}}.
\label{crudeaim}
\end{equation}
First note that, since $r_{A:A}(x)\leq{|A|}$ for any $x$,
\begin{equation}\label{f:main3_E_E_3}
    \E^*_3(A)=\sum_{x\in{A:A}}r_{A:A}^3(x)\leq{|A|\sum_{x\in{A:A}}r_{A:A}^2(x)}=|A|\E^*(A) \,,
\end{equation}
so that \eqref{soly3} yields
\begin{equation}
\E^*_3(A)\ll{|A||A+A|^2\log{|A|}}.
\label{soly4}
\end{equation}
Now, apply Corollary \ref{E3}, with $B=A$ and $C=A+A$. We obtain
$$\E^*_3(A)|A(A+A+A)|^4\gg{\frac{|A|^8|A+A|^2}{(\log{|A|})^2}}.$$
Combining this with the upper bound on $\E^*_3(A)$ from \eqref{soly4}, it follows that
$$|A(A+A+A)|\gg{\frac{|A|^{7/4}}{(\log{|A|})^{3/4}}},$$
which proves \eqref{crudeaim}.

Now, we will show how a slightly more subtle argument can lead to a small improvement in this exponent. Apply \eqref{soly3} and Lemma \ref{mainlemma}, with $B=A$ and $C=A+A$, so that
\begin{equation}\label{tmp:29.10.2013_1}
    |A|^5|A+A|\lessapprox{\E^*(A)|A(A+A+A)|^2}\lessapprox{|A+A|^2|A(A+A+A)|^2} \,,
\end{equation}
and thus
\begin{equation}
|A+A||A(A+A+A)|^2\gtrapprox{|A|^5}.
\label{first}
\end{equation}
Write $\E^*(A)=\frac{|A|^3}{K}$, for some value $K\geq{1}$. By the first inequality from (\ref{tmp:29.10.2013_1}), it follows that
\begin{equation}
|A(A+A+A)|\gtrapprox{|A|K^{1/2}|A+A|^{1/2}} \,.
\label{second}
\end{equation}
Applying Solymosi's bound for the multiplicative energy then yields
\begin{equation}
|A(A+A+A)|\gtrapprox{|A|^{7/4}K^{1/4}}.
\label{second2}
\end{equation}
Now, by Theorem \ref{BSG} there exists a subset $A'\subset{A}$ such that
\begin{equation}
|A'|\gtrapprox{\frac{|A|}{K}}
\label{BSG1}
\end{equation}
and
\begin{equation}
|A':A'|\lessapprox{K^4\frac{|A'|^3}{|A|^2}}.
\label{BSG2}
\end{equation}
By Theorem \ref{convex} and \eqref{BSG2},
\begin{align*}
|A'|^{24}&\lessapprox{|A'+A'|^{9}|A':A'|^{10}}
\\&\ll{|A+A|^9K^{40}\frac{|A'|^{30}}{|A|^{20}}},
\end{align*}
and then
$$|A+A|^9\gtrapprox{\frac{|A|^{20}}{|A'|^6K^{40}}}\geq{\frac{|A|^{14}}{K^{40}}}.$$
From the latter inequality we now have $|A+A|\gtrapprox{\frac{|A|^{14/9}}{K^{40/9}}}$. Comparing this with \eqref{second} leads to the following bound:
\begin{equation}
|A(A+A+A)|\gtrapprox{\frac{|A|^{16/9}}{K^{31/18}}}.
\label{fourth}
\end{equation}
The worst case occurs when $K\approx{|A|^{1/71}}$. It can be verified that if $K<|A|^{1/71}$, then
$$|A(A+A+A)|\gtrapprox{|A|^{\frac{7}{4}+\frac{1}{284}}},$$
by inequality \eqref{fourth}. On the other hand, if $K\geq{|A|^{1/71}}$, then it follows from inequality
\eqref{second2} that
$$|A(A+A+A)|\gtrapprox{|A|^{\frac{7}{4}+\frac{1}{284}}}.$$
Therefore, we have proved that \eqref{A(A+A+A)} holds for all $K$ (i.e. for all possible values of $\E^*(A)$), which concludes the proof.

\begin{flushright}
\qedsymbol
\end{flushright}

\section{Proofs of Results on Products of Translates}
We record a short lemma which will be used in the proofs of Theorem \ref{translates2} and \ref{translates3}
\begin{lemma} \label{handylemma}
Let $A\subset{\mathbb{R}}$ be a finite set. Then, for any $x\in{\mathbb{R}}$,
$$|(A+x)(A+A)||A+A|\gg{\frac{|A|^3}{\log{|A|}}}.$$
\end{lemma}

\begin{proof} Note that for any $x\in{\mathbb{R}}$
\begin{align}
|A+A|^2&=|(A+x)+(A+x)|^2
\\&\gg{\frac{\E^*(A+x)}{\log{|A|}}} \label{al1}
\\&\gg{\frac{|A|^6}{|(A+x)(A+A)|^2(\log{|A|})^2}}, \label{al2}
\end{align}
where \eqref{al1} is an application of Solymosi's bound \eqref{soly3}, and \eqref{al2} comes from Lemma \ref{mainlemma}. The lemma follows after rearranging this expression.
\end{proof}

\subsection*{Proof of Theorem \ref{translates1}}
Recall that Theorem \ref{translates1} states that 
$$|A(A+a)|\gg{|A|^{3/2}}$$
holds for at least half of the elements $a$ belonging to $A$. Lemma \ref{sum*} tells us that, for some fixed constant $C$
$$\sum_{a\in{A}}\E^*(A,a+A)\leq{C|A|^{7/2}}.$$
Let $A'\subset{A}$ be the set
$$A':=\{a\in{A} ~:~ \E^*(A,a+A)\leq{2C|A|^{5/2}}\},$$
and observe that
$$2C|A|^{5/2}|A\setminus{A'}|\leq{\sum_{a\in{A\setminus{A'}}}\E^*(A,a+A)}\leq{C|A|^{7/2}},$$
which implies that
$$|A\setminus{A'}|\leq{\frac{|A|}{2}}.$$
This implies that $|A'|\geq{\frac{|A|}{2}}$. To complete the proof, we will show that for every $a\in{A'}$ we have $|A(A+a)|\gg{|A|^{3/2}}$. To see this, simply observe that, for any $a\in{A'}$,
$$\frac{|A|^4}{|A(A+a)|}\leq{\E^*(A,A+a)}\ll{|A|^{5/2}}.$$
The lower bound here comes from \eqref{CS}, whilst the upper bound comes from the definition of $A'$. Rearranging this inequality gives
$$|A(A+a)|\gg{|A|^{3/2}},$$
as required.
\begin{flushright}
\qedsymbol
\end{flushright}

We remark that it is straightforward to adapt this argument slightly---switching the roles of addition and multiplication and using Lemma \ref{sum+} in place of Lemma \ref{sum*}---in order to show that there exists a subset $A'\subset{A}$, such that $|A'|\geq{\frac{|A|}{2}}$, with the property that
$$|A+aA|\gg{|A|^{3/2}},$$
for any $a\in{A'}$.

It is also easy to adapt the proof of Theorem \ref{translates1} in order to show that, for any $0<\epsilon<1$ and any $A\subset{\mathbb{R}}$, there exists a subset $A'\subset{A}$ such that $|A'|\geq{(1-\epsilon)|A|}$, and for all $a\in{A'}$,
$$|A(A+a)|\gg_{\epsilon}{|A|^{3/2}}.$$
In other words, the set $A(A+a)$ is large for all but a small positive proportion of elements $a\in{A}$. The analogous statement for $A+aA$ is also true.

\subsection*{Proof of Theorem \ref{translates2}}
Recall that Theorem \ref{translates2} states that  
$$|(A+a)(A+A)|\gg{\frac{|A|^{5/3}}{(\log{|A|})^{1/3}}}$$
holds for at least half of the elements $a$ belonging to $A$.
This proof is similar to the proof of Theorem \ref{translates1}. Again, Lemma \ref{sum*} tells us that for a fixed constant $C$, we have
$$\sum_{a\in{A}} \E^*(A+A,a+A)\leq{C|A|^2|A+A|^{3/2}}.$$
Define $A'\subset{A}$ to be the set
$$A':=\{a\in{A} ~:~ \E^*(A+A,a+A)\leq{2C|A||A+A|^{3/2}} \},$$
and observe that
$$2C|A||A+A|^{3/2}|A\setminus{A'}|\leq{\sum_{a\in{A\setminus{A'}}} \E^*(A+A,a+A)}\leq{C|A|^2|A+A|^{3/2}}.$$
This implies that $|A\setminus{A'}|\leq{\frac{|A|}{2}}$, and so
$$|A'|\geq{\frac{|A|}{2}}.$$
Next, observe that, for any $a\in{A'}$,
$$\frac{|A|^2|A+A|^2}{|(A+a)(A+A)|}\leq{\E^*(A+A,A+a)}\ll{|A||A+A|^{3/2}}.$$
The lower bound here comes from \eqref{CS}, whilst the upper bound comes from the definition of $A'$. After rearranging, we have
\begin{equation}
|(A+a)(A+A)|\gg{|A||A+A|^{1/2}},
\label{nearly}
\end{equation}
for any $a\in{A'}$. To complete the proof we need a useful lower bound on $|A+A|$. This comes from Lemma \ref{handylemma}, which tells us that for any $a\in{\mathbb{R}}$, and so certainly any $a\in{A}$,
$$|A+A|^{1/2}\gg{\frac{|A|^{3/2}}{(\log{|A|})^{1/2}|(A+a)(A+A)|^{1/2}}}.$$
Finally, this bound can be combined with \eqref{nearly}, to conclude that
$$|(A+a)(A+A)|\gg{\frac{|A|^{5/3}}{(\log{|A|})^{1/3}}},$$
as required.
\begin{flushright}
\qedsymbol
\end{flushright}

\subsection*{Another upper bound on the multiplicative energy}

\newcommand{\differentfont}[1]{\ensuremath{#1}}

Before proceeding to the proof of Theorem \ref{translates3}, it is necessary to establish another upper bound on the multiplicative energy. This is essentially a calculation, based on earlier work from \cite{GS} and \cite{TimORN}. We will need the following lemma:

\begin{lemma} \label{multiplicity} Suppose that $\differentfont{A},\differentfont{B}$ and $\differentfont{C}$ are finite subsets of $\mathbb{R}$\/ such that $0\not\in A,B$, and $\alpha\in{\mathbb{R}\setminus{\{0\}}}$. Then, for any integer $t\geq{1}$,
$$|\{s:r_{\differentfont{A}\differentfont{B}}(s)\geq{t}\}|\ll{\frac{|(\differentfont{A}+\alpha)\cdot{\differentfont{C}}|^2|\differentfont{B}|^2}{|\differentfont{C}|t^3}}.$$
\end{lemma}

This statement is a slight generalisation of Lemma 3.2 in \cite{TimORN}. We give the proof here for completeness.

\begin{proof} For some values $p$ and $b$, define the line $l_{p,b}$ to be the set $\{(x,y):y=(px-\alpha)b\}$.
Let $\differentfont{L}$ be the family of lines
$$\differentfont{L}:=\{l_{p,b}:p\in{(\differentfont{A}+\alpha)\differentfont{C}},b\in{\differentfont{B}}\}.$$
Observe that, since $\alpha$ is non-zero, $|\differentfont{L}|=|(\differentfont{A}+\alpha)\differentfont{C}||\differentfont{B}|.$\footnote{Note that it is not true in general that $|L|=|(A+\alpha)C||B|$. Indeed, if $0\in{B}$, then $l_{p,0}=l_{p',0}$ for $p\neq{p'}$, and so the lines may not all be distinct. However, we may assume again that zero does not cause us any problems. To be more precise, we assume that $0\notin{\differentfont{B}}$, as otherwise $0$ can be deleted, and this will only slightly change the implied constants in the statement of the lemma. If $0\notin{B}$, then the statement that $|L|=|(A+\alpha)C||B|$ is true.} Let $P_t$ denote the set of all $t$-rich points in the plane. By Corollary \ref{STcor}, for $t\geq{2}$,

\begin{equation}
|P_t|\ll{\frac{|\differentfont{B}|^2|(\differentfont{A}+\alpha)\differentfont{C}|^2}{t^3}+\frac{|\differentfont{B}||(\differentfont{A}+\alpha)\differentfont{C}|}{t}},
\label{ST3}
\end{equation}

and it can once again be simply assumed that

\begin{equation}
|P_t|\ll{\frac{|\differentfont{B}|^2|(\differentfont{A}+\alpha)\differentfont{C}|^2}{t^3}}.
\label{ST4}
\end{equation}

This is because, if the second term from \eqref{ST3} is dominant, it must be the case
$$t>|(\differentfont{A}+\alpha)\differentfont{C}|^{1/2}|\differentfont{B}|^{1/2}\geq{\min{\{|\differentfont{A}|,|\differentfont{B}|\}}}.$$
However, in such a large range, $|\{s:r_{\differentfont{AB}}(s)\geq{t}\}|=0$, and so the statement of the lemma is trivially true.

Next, it will be shown that for every $s\in{\{s:r_{\differentfont{AB}}(s)\geq{t}\}}$, and for every element $c\in{\differentfont{C}}$,
\begin{equation}
\left(\frac{1}{c},s\right)\in{P_t}.
\label{finally}
\end{equation}

Once, \eqref{finally} has been established, it follows that $|P_t|\geq{|\differentfont{C}||\{s:r_{\differentfont{AB}}(s)\geq{t}\}|}$. Combining this with \eqref{ST4}, it follows that
\begin{equation}
|\{s:r_{\differentfont{AB}}(s)\geq{t}\}|\ll{\frac{|\differentfont{B}|^2|(\differentfont{A}+\alpha)\cdot{\differentfont{C}}|^2}{|\differentfont{C}|t^3}},
\label{SThere}
\end{equation}
for all $t\geq{2}$. We can then check that \eqref{SThere} is also true in the case when $t=1$, since
$$\frac{|B|^2|(A+\alpha)C|^2}{1^3|C|}\geq{|B|^2|(A+\alpha)C|}\geq{|AB|}=|\{s:r_{AB}(s)\geq{1}\}|.$$

It remains to establish \eqref{finally}. To do so, fix $s$ with $r_{\differentfont{AB}}(s)\geq{t}$ and $c\in{\differentfont{C}}$. The element $s$ can be written in the form $s=a_1b_1=\cdots=a_tb_t$. For every $1\leq{i}\leq{t}$ we have
\begin{align*}
s&=a_ib_i
\\&=(a_i+\alpha-\alpha)b_i
\\&=\left(\frac{(a_i+\alpha)c}{c}-\alpha\right)b_i,
\end{align*}
which means that $\left(\frac{1}{c},s\right)$ belongs to the line $l_{(a_i+\alpha)c,b_i}$. As $i$ varies from $1$ through to $t$ this is still true, and it is also true that the lines $l_{(a_i+\alpha)c,b_i}$ are distinct for all values of $i$ in this range. Therefore, $\left(\frac{1}{c},s\right)\in{P_t}$, as claimed. This concludes the proof.

\end{proof}

We use this to prove another lemma:

\begin{lemma} \label{etranslates} For any finite sets $\differentfont{A}$ and $\differentfont{C}$ in $\mathbb{R}$, and any $\alpha\in{\mathbb{R}\setminus{\{0\}}}$,
$$\E^*(\differentfont{A})\ll{\frac{|(\differentfont{A}+\alpha)\differentfont{C}||\differentfont{A}|^2}{|\differentfont{C}|^{1/2}}}.$$
\end{lemma}

\begin{proof} Let $\triangle\geq{1}$ be a fixed integer to be specified later. Observe that,
$$\E^*(\differentfont{A})=\sum_{x}r_{\differentfont{A}:\differentfont{A}}^2(x)=\sum_{x\,:\,r_{\differentfont{A}:\differentfont{A}}(x) \le \triangle}r_{\differentfont{A}:\differentfont{A}}^2(x)+\sum_{x:r_{\differentfont{A}\,:\,\differentfont{A}}(x) > {\triangle}}r_{\differentfont{A}:\differentfont{A}}^2(x).$$
The first term is bounded by
$$\sum_{x\,:\,r_{\differentfont{A}:\differentfont{A}}^2(x) \le \triangle}r_{\differentfont{A}:\differentfont{A}}^2(x)\leq{\triangle\sum_{x}r_{\differentfont{A}:\differentfont{A}}(x)}=\triangle|\differentfont{A}|^2.$$
For the second term, apply a dyadic decomposition and use Lemma \ref{multiplicity} as follows:
\begin{align*}
\sum_{x:r_{\differentfont{A}:\differentfont{A}}(x)>{\triangle}}r_{\differentfont{A}:\differentfont{A}}^2(x)&=\sum_{j}\sum_{x\,:\,2^{j-1}\triangle<{r_{\differentfont{A}:\differentfont{A}}(x)} \le 2^j\triangle}r_{\differentfont{A}:\differentfont{A}}^2(x)
\\&\ll{\sum_{j}(2^j\triangle)^2\frac{|(\differentfont{A}+\alpha)\differentfont{C}|^2|\differentfont{A}|^2}{|\differentfont{C}|(2^j\triangle)^3}}
\\&=\frac{|(\differentfont{A}+\alpha)\differentfont{C}|^2|\differentfont{A}|^2}{|\differentfont{C}|\triangle}.
\end{align*}
This shows that
$$\E^*(\differentfont{A})\ll{\triangle|\differentfont{A}|^2+\frac{|(\differentfont{A}+\alpha)\differentfont{C}|^2|\differentfont{A}|^2}{|\differentfont{C}|\triangle}},$$
and we optimise by setting $\triangle=\left\lceil\frac{|(\differentfont{A}+\alpha)\differentfont{C}|}{|\differentfont{C}|^{1/2}}\right\rceil\approx{\frac{|(A+\alpha)C|}{|C|^{1/2}}}$. This shows that
$$\E^*(\differentfont{A})\ll{\frac{|(\differentfont{A}+\alpha)\differentfont{C}||\differentfont{A}|^2}{|\differentfont{C}|^{1/2}}},$$
as required.
\end{proof}

\subsection*{Proof of Theorem \ref{translates3}}

Let $a$ and $b$ be distinct real numbers. We will show that

\begin{equation}
|(A+a)(A+A)|^5|(A+b)(A+A)|^2\gg{\frac{|A|^{11}}{(\log{|A|})^3}}.
\label{aim2}
\end{equation}

Once we have established \eqref{aim2}, the theorem follows, since this implies that for any $a,b\in{\mathbb{R}}$ with $a\neq{b}$, we have
\begin{equation}
\max{\{|(A+a)(A+A)|,|(A+b)(A+A)|\}}\gg{\frac{|A|^{11/7}}{(\log{|A|})^{3/7}}},
\label{aim3}
\end{equation}
and therefore, there may exist at most one value $x\in{\mathbb{R}}$ which violates the inequality
$$|(A+x)(A+A)|\gg{\frac{|A|^{11/7}}{(\log{|A|})^{3/7}}}.$$
It remains to prove \eqref{aim2}. First, apply Lemma \ref{etranslates} with $A=A+a$, $C=A+A$ and $\alpha=b-a\neq{0}$. This yields  
\begin{align}
\E^*(A+a)&\ll{\frac{|(A+a+(b-a))(A+A)||A|^2}{|A+A|^{1/2}}}
\\&=\frac{|(A+b)(A+A)||A|^2}{|A+A|^{1/2}}. \label{upper}
\end{align}
Meanwhile, Lemma \ref{mainlemma} informs us that
\begin{equation}
\E^*(A+a)\gg{\frac{|A|^6}{|(A+a)(A+A)|^2\log{|A|}}},
\label{lower}
\end{equation}
and combining \eqref{upper} and \eqref{lower}, we have
\begin{equation}
|(A+a)(A+A)|^2|(A+b)(A+A)|\gg{\frac{|A|^4|A+A|^{1/2}}{\log{|A|}}}.
\label{nearly2}
\end{equation}
Finally, we apply Lemma \ref{handylemma} which tells us that
$$|A+A|^{1/2}\gg{\frac{|A|^{3/2}}{(\log{|A|})^{1/2}|(A+a)(A+A)|^{1/2}}}.$$
Plugging this bound into \eqref{nearly2} and rearranging, it follows that
$$|(A+a)(A+A)|^5|(A+b)(A+A)|^2\gg{\frac{|A|^{11}}{(\log{|A|})^3}}.$$
Thus we have established \eqref{aim2}, which completes the proof.

\begin{flushright}
\qedsymbol
\end{flushright}

\subsection*{Proof of Theorem \ref{t:main_intr_II}}
\label{sec:further}
 
Before we prove Theorem \ref{t:main_intr_II}, we need some auxiliary statements.
First we note a corollary of the proof of Lemma \ref{mainlemma}.
\begin{corollary} 
  \label{mainlemmacor}
  Let $A,B,$ and $C$ be finite subsets of $\mathbb{R}$, and let
\[
S^\star=|\{(a,b,c,a',b',c')\in(A\times B\times C)^2\colon a(b+c)=a'(b'+c')\not=0\}| 
\]
Then
\begin{equation}
S^\star=  \sum_{x\not=0}r_{A(B+C)}^2(x)\ll \E_2^*(A)^{1/2}|B|^{3/2}|C|^{3/2}(\log|A|)^{1/2}.
\end{equation}
\end{corollary}

Now recall a lemma from \cite{SS3}.

\begin{lemma}
\label{corpop}
    Let $A$ be a subset of an abelian group, $P_* \subseteq A-A$ and $\sum_{s\in P_*} |A_s| = \eta |A|^2$, $\eta \in (0,1]$.
    Then
    \begin{equation*}
        \sum_{s\in P_*} |A\pm A_s| \geq \frac{\eta^2 |A|^6}{\E_3(A) } \,.
    \end{equation*}
\end{lemma}

\begin{corollary}
    Let $A$ be a subset of an abelian group.
    Then
\begin{equation}\label{f:E(A+A)_1}
    \E(A+A) \ge |A-A|^{-1} |A\m A + \D(A)|^2 \ge |A|^2 \max\{ |A-A|, |A+A| \} \,,
\end{equation}
\begin{equation}\label{f:E(A+A)_1.5}
    \E(A-A) \ge |A-A|^{-1} |A\m A - \D(A)|^2 \ge |A|^2 |A-A| \,,
\end{equation}
    and
\begin{equation}\label{f:E(A+A)_2}
    \E(A \pm A) \ge \frac{|A|^{12}}{\E^2_3 (A) |A-A|} \,.
\end{equation}
\label{c:E(A+A)}
\end{corollary}
\begin{proof}
We prove the statement for sums, the result for differences can be obtained similarly.
Put $S=A+A$ and $D=A-A$.
By Katz--Koester trick \cite{kk}, we get
$$
    |(A + A) \cap (A + A - s)| \ge |A + A_{s}| \,,
$$
and
\[
    |(A - A) \cap (A - A - s)| \ge |A - (A\cap (A+s))| \,.
\]
Thus by the Cauchy--Schwarz inequality
\begin{align*}
  \E(S) &\ge \sum_{s\in D} r_{S-S}^2 (s) \ge \sum_{s\in D} |A + A_s|^2 \ge |D|^{-1} \left( \sum_{s\in D} |A+A_s| \right)^2\\
  &=|D|^{-1} |A \m A + \D(A)|^2\ge |A|^2 |A-A| \,,
\end{align*}
and, similarly,
\begin{align*}
  \E(S) &\ge |D|^{-1} |A \m A + \D(A)|^2
  \ge |D|^{-1} \left( \sum_{x\in S} |A+(A\cap (x-A))| \right) \left( \sum_{s\in D} |A+A_s| \right)\\
  &\ge |A|^2 |A+A|
\end{align*}
as required.
Here we have used the fact
$$
    |A \m A + \D(A)| = \sum_{s\in D} |A+A_s| = \sum_{x\in S} |A+(A\cap (x-A))| \,,
$$
which follows from the consideration of the projections of the set $A \m A + \D(A)$.
More precisely, one has $A \m A + \D(A) = \{ (a_1+a,a_2+a) ~:~ a,a_1,a_2\in A \}$.
Whence, writing $s=(a_1+a) - (a_2+a) = a_1-a_2 \in D$, we get $a_2 \in A_s$, $a+a_2 \in A+A_s$ and viceversa.
Similarly, put $x=a_1+a_2 \in S$, one get  $a_2 \in A\cap (x-A)$, $a+a_2 \in A+(A\cap (x-A))$ and viceversa.

Further, by Lemma \ref{corpop}
$$
    |A|^6 \le \E_3 (A) \sum_x D(x) r_{S-S} (x) \,.
$$
Applying the Cauchy--Schwarz inequality, we get
$$
    |A|^{12} \le \E^2_3 (A) \E (S) |D|
$$
and formula (\ref{f:E(A+A)_2}) follows.
The result for the set $D$ is similar.
\end{proof}

Finally, we can prove Theorem \ref{t:main_intr_II}:
\begin{proof}[Proof of Theorem \ref{t:main_intr_II}]
  We begin with the first formula of the result.
   
Take $C=A-B$ in Corollary \ref{mainlemmacor}.
Note that $r_{(A-B)+B}(a)\geq |B|$ for all $a\in A$, which implies that $r_{A(B+C)}(x)\geq r_{AA}(x)|B|$.
Thus by Corollary \ref{mainlemmacor} we have\footnote{Note that $r_{AA}(x)=0$ for $x=0$ since we have assumed that $0\not\in A$.}
\[
|B|^2\E^*_2(A)\leq\sum_{x\not=0}r_{A(B+C)}^2(x)\ll \E^*_2(A)^{1/2}|B|^{3/2}|A-B|^{3/2}(\log|A|)^{1/2}.
\]
Rearranging and applying the Cauchy-Schwarz lower bound for $\E_2^*(A)$ yields
\[
\frac{|A|^4|B|}{|AA^{\pm 1}|}\leq |B|\E_2^*(A)\ll |A-B|^3\log|A|,
\]
as required.

Combining (\ref{f:main_intr_2_new}) with Corollary \ref{c:E(A+A)}, we obtain (\ref{f:main_intr_2'-_new}).
 This
 completes the proof.
\end{proof}


\section*{Concluding remarks - the complex case}
We conclude by pointing out that almost all of the results in this paper also hold in the more general case whereby $A$ is a finite set of complex numbers, since the tools we have made use of can all be extended in this direction. Indeed, the Szemer\'{e}di-Trotter was extended to points and lines in $\mathbb{C}^2$ by Toth \cite{toth}. More modern proofs have recently appeared due to Zahl \cite{zahl}, and Solymosi-Tao \cite{solytao}, although the latter of these results has exponents which are infinitesimally worse. The other main tool which has been imported to this paper is Solymosi's \cite{solymosi} bound on the multiplicative energy (which we earlier labelled \eqref{soly3}). This result was recently extended to the case when $A\subset{\mathbb{C}}$ by Konyagin and Rudnev \cite{KR}.
 

\section*{Acknowledgements}
We would like to thank Antal Balog and Tomasz Schoen for several helpful conversations, and Misha Rudnev for helping to significantly simplify the proof of Lemma \ref{mainlemma}. We are grateful to M. Z. Garaev for informing us about Theorem \ref{BSG}.

\noindent{Department of Mathematics\\ University of Rochester\\
Rochester, NY 14620\\
{\tt murphy@math.rochester.edu}

\noindent{Department of Mathematics and Statistics,\\ University of Reading,\\
Whiteknights, Reading\\
RG6 6AX\\
{\tt o.rochenewton@gmail.com}

\noindent{Division of Algebra and Number Theory,\\ Steklov Mathematical
Institute,\\
ul. Gubkina, 8, Moscow, Russia, 119991\\
and
\\
IITP RAS,  \\
Bolshoy Karetny per. 19, Moscow, Russia, 127994\\}
{\tt ilya.shkredov@gmail.com}

\end{document}